%% file: 00_Computing_T_opt_Designs.tex
\pdfoutput=1

\documentclass[11pt,a4paper]{article}

\usepackage{geometry}
\usepackage[utf8]{inputenc}
\usepackage{fontenc}
\usepackage[english]{babel}
\usepackage{amsmath}
\usepackage{amsfonts}
\usepackage{amssymb}
\usepackage{amsopn}
\usepackage{amsthm}
\usepackage{mathtools}
\usepackage{graphicx, caption}
\usepackage{algorithm}
\usepackage{algorithmicx}
\usepackage{algpseudocode}
\usepackage[hidelinks]{hyperref}
\usepackage{nicefrac}
\usepackage{bbm}
\usepackage{mhchem}
\usepackage{subcaption}
\usepackage{setspace}
\usepackage{siunitx}
\usepackage{tabularx, booktabs}
\usepackage{tikz}
\usepackage{todonotes}
\usepackage[title]{appendix}

\newtheorem{theorem}{Theorem}
\newtheorem{prop}[theorem]{Proposition}
\newtheorem{lem}[theorem]{Lemma}
\newtheorem{cor}[theorem]{Corollary}

\theoremstyle{definition}
\newtheorem{defi}[theorem]{Definition}
\newtheorem{bem}[theorem]{Remark}
\newtheorem{examplex}[theorem]{Example}
\newtheorem{nota}[theorem]{Notation}
\newtheorem{assum}[theorem]{Assumption}

\newcommand{\para}{\theta}
\newcommand{\Para}{\Theta}
\DeclareMathOperator{\X}{\mathcal{X}}
\DeclareMathOperator{\BigO}{\mathcal{O}}
\newcommand{\w}{w}
\DeclareMathOperator{\Y}{\mathcal{Y}}
\newcommand{\f}{f}

\DeclareMathOperator{\supp}{\text{supp}}

\DeclareMathOperator*{\argmax}{arg\,max}
\DeclareMathOperator*{\argmin}{arg\,min}

\newcommand{\high}[1]{^{(#1)}}

\algnewcommand{\parState}[1]{\State
	\parbox[t]{\dimexpr\linewidth-\algmargin}{\strut #1\strut}}

\title{Computing T-optimal designs via nested semi-infinite programming and twofold adaptive discretization}

\author{David Mogalle
	\footnote{Fraunhofer Institute for Industrial Mathematics (ITWM), Fraunhofer-Platz 1, 67663 Kaiserslautern, Germany}
	\footnote{Corresponding author: David.Mogalle@itwm.fraunhofer.de} , 
	Philipp Seufert$^*$,
	Jan Schwientek$^*$,
	Michael Bortz$^*$,
	Karl-Heinz K\"{u}fer$^*$}

\begin{document}
	
	\maketitle
	
	\begin{abstract}
		Modeling real processes often results in several suitable models. In order to be able to distinguish, or discriminate, which model best represents a phenomenon, one is interested, e.g., in so-called T-optimal designs. These consist of the (design) points from a generally continuous design space at which the models deviate most from each other, under the condition that they are best fitted to those points. Thus, the T-criterion represents a bi-level optimization problem, which can be transferred into a semi-infinite one, but whose solution is very unstable or time consuming for non-linear models and non-convex lower- and upper-level problems.
		
		If one considers only a finite number of possible design points, a numerically well tractable linear semi-infinite optimization problem arises. Since this is only an approximation of the original model discrimination problem, we propose an algorithm which alternately and adaptively refines discretizations of the parameter as well as of the design space and, thus, solves a sequence of LSIPs. We prove convergence of our method and its subroutine and show on the basis of discrimination tasks from process engineering that our approach is stable and can outperform the known methods.
	\end{abstract}

	
    \vspace{5pt}
    
	
	\input{Introduction/Introduction}

	\input{T-Criterion/The_T-Criterion}

	\input{2-ADAPT-MD/The_2-ADAPT-MD_Algorithm}

	\input{Convergence/Convergence}

	\input{NumericalResults/NumericalResults}

	\input{Conclussion/Conclussion}
	
	\bibliographystyle{abbrv}
	\bibliography{lit_MA}
	
	\newpage
	\begin{appendices}
		\input{Appendix/SIP}
	\end{appendices}

\end{document}

%% file: Introduction/Introduction.tex
\section{Introduction}
\label{sec:intro}

\subsection{Motivation}

In science, one uses mathematical formulas to describe processes or phenomena. For example in chemical process engineering there exists a plethora of models that describe the course of a reaction.

Often one can make an educated guess based on experience or prior knowledge which model fits an investigated phenomenon. Still it might happen that there are several plausible models. In order to decide which of these models fits the true underlying phenomenon best, we want to conduct further experiments. Then, based on these additional experiments, we can hopefully clearly distinguish, i.e.\ discriminate, the suggested models and pick the one that suits best.

However, such experiments can be very time consuming and costly. For example, a large batch reactor needs time and energy until it reaches the right temperature and pressure to conduct an experiment. Therefore, if we carry out costly experiments, we want to do as few experiments as possible. Thus, every chosen experiment should carry as much information as attainable.

The task of determining such experiments is generally called \textit{design of experiments (DoE)}. In this article we focus on \textit{model discrimination (MD)}, which is a sub-problem in the family of DoE problems: Given two plausible models $f_1$ and $f_2$, as well as perhaps some existing data $\mathcal{D}_N$, we want to find settings for new experiments that give insight whether $f_1$ or $f_2$ better fits the real phenomenon.

\subsection{Literature}

In order to evaluate whether an experiment is worthwhile or not, several criteria have been introduced. Consequently, optimal settings for new experiments can be determined by optimizing these criteria.

For model discrimination specifically, the first recorded criterion was presented by Hunter and Reiner \cite{hunter1965designs}. They proposed fitting the two competing response models $\f_1$ and $\f_2$ to some given data $\mathcal{D}_N$, and then performing the experiment where the models differ the most after the fit.

One issue with the Hunter-Reiner criterion is its implicit assumption that the measurement errors for both models have constant variance. To circumvent this problem, Box and Hill \cite{box1967discrimination} proposed a Bayesian sequential discrimination method based on the \textit{Kullback-Leibler divergence} (short: KL-divergence, see, e.g., \cite{bishop2006pattern}). Additionally, they formulated their method in a way such that several competing models can be discriminated.

Then in 1975, Atkinson and Fedorov \cite{atkinson1975a} introduced the \textit{T-optimality criterion}, or short \textit{T-criterion}. The idea for this criterion is to fix one of the models, while fitting the other model as good as possible to the first one by a weighted least-squares sum. Afterwards, one selects the points with the largest difference between the models as the proposed experiments. 

In the same year, Atkinson and Fedorov also introduced an extension for discriminating several models at once \cite{atkinson1975b}.

Theoretic properties of the T-criterion have been investigated in further publications. A good first overview of the most important characteristics can be found in \cite{FedDoe2013}. A comprehensive analysis is included in the dissertation by Kuczewski \cite{kuczewski2006diss}. Additionally, \cite{dette2009titoff} states the duality of the T-criterion and the Chebychev approximation problem, which was later used in \cite{braess2013optimal} to further characterize T-optimal solutions.

An extension of the T-criterion for discriminating dynamic systems that depend on end times and initial values has been introduced in \cite{ucinski2005t}. Recently, a further generalization of the T-criterion for semi-parametric models has been presented by \cite{dette2018optimal}.

An important advancement of the T-criterion was introduced by López-Fidalgo et al.\ \cite{lopez2007optimal}: the \textit{KL-optimality criterion}. In fact, the idea for the KL-criterion is the same as for the T-criterion, but in order to fit one model to the other they use the Kullback-Leibler divergence instead of a sum of least-squares. The advantage of this criterion is that we can deal with models that assume non-Gaussian measurement errors, which the T-criterion cannot do out of the box. 

In order to compute optimal solutions for the T-criterion, called \textit{T-optimal designs}, Atkinson and Fedorov \cite{atkinson1975a} used an algorithm that is based on the Vector Direction Method \cite{FedDoe2013, WynnAlgorithm}. But since then, many more strategies have been successfully applied. Kuczewski \cite{kuczewski2006diss} and Duarte et al. \cite{duarte2015semi} both formulated the T-criterion as a semi-infinite program (SIP) and solved it using the general Blankenship \& Falk algorithm. Braess and Dette \cite{braess2013optimal} used results from approximation theory to formulate an algorithm which alternatingly finds new points by solving the dual Chebyshev approximation problem and calculates new weights for a new solution of the primal T-criterion. Later, Dette et al.\ \cite{dette2015bayesian} simplified this approach by finding several new design points using the directional derivative from the Vector Direction Method, and computing the corresponding new weights for the current selection of design points with the help of a linearization. A heuristic approach based on particle swarm optimization has been successfully applied by Chen et al. \cite{chen2020hybrid}.

In the special case where we discriminate two polynomial models, Yue et al. \cite{yue2019t} presented an approach based on moment relaxation and semi-definite programming to find T-optimal solutions. In the even more special case where the degree of the two competing polynomial models differ by exactly two, Dette et al. \cite{dette2012t} were even able to give an analytic approach for calculating T-optimal designs. Additional investigation when discriminating polynomial, or rational models has been done by \cite{guchenko2017t}.

\subsection{Outline}

In this paper, we present a new algorithm for computing T-optimal designs based on two nested adaptive discretization schemes. At first, we introduce some preliminaries in Section \ref{sec:prelim}. This includes basic notations from design theory, and selected insights on the T-criterion. Then, we introduce our new algorithm in Section \ref{sec:alg}. This section is separated into two parts, where the first considers the inner discretization scheme, while the second considers the outer one. We also show that our algorithm convergences towards arbitrary good approximations of T-optimal designs under realistic assumptions in Section \ref{sec:conv}. Finally in Section \ref{sec:num}, we compare our algorithm against three other state-of-the-art methods on two examples from chemical process engineering.

%% file: T-Criterion/The_T-Criterion.tex
\section{Preliminaries}
\label{sec:prelim}

We start by briefly introducing some basic notations from design theory and presenting the main characteristics for T-optimal experimental designs.

\subsection{Fundamentals of Design Theory}

A model function in the following context is of the form $\f: \X\times\,\Para\to\Y$. It maps \textit{inputs} or \textit{design points} $x$ to their respective \textit{outputs} or \textit{responses} $y$ and can be tweaked by parameters $\para$. 

We call $\X$ the \textit{design space}, $\Y$ the \textit{output space}, and $\Para$ the \textit{parameter space}.
These sets are assumed to be subsets of the multi-dimensional real numbers, i.e.\ $\Para\subseteq\mathbb{R}^{d_\para}$, $\X\subseteq\mathbb{R}^{d_x}$, and $\Y\subseteq\mathbb{R}^{d_y}$.
Additionally, if the number of outputs is one, i.e.\ $d_y = 1$, we call $f$ a \textit{single-response} model, whereas if it has more than one output, i.e.\ $d_y > 1$, then we call $f$ a \textit{multi-response} model.

An experiment in this context coincides with a design point $x\in\X$. This point gives the settings for the corresponding experiment, like temperature, pressure, initial concentrations, etc.

As introduced in \cite{FedDoe2013}, we model an \textit{experimental design plan} as a probability measure. In case of a discrete measure, we write
$$\xi = \left\lbrace \begin{array}{ccc}
	x_1 & \ldots & x_N \\ 
	w_1 & \ldots & w_N
\end{array}  \right\rbrace,$$
with
\[w_i\geq0,\qquad\Vert\w\Vert_1 = \sum_{i=1}^N w_i = 1.\]
Here, the $x_i\in\X$ are the suggested experiments, and the weights $w_i$ represent the relative number of experiments being proposed for their respective design points. The \textit{set of experimental designs}, which coincides with the set of probability measures on $\X$, is denoted by $\Xi$, or $\Xi(\X)$.

%


\subsection{The T-criterion}

We now rigorously introduce the so-called T-optimality criterion, which aims to find several design points with maximum squared difference between two competing models $f_1$ and $f_2$.

\begin{defi}[T-Optimality Criterion]
	\label{defi:Topt}
	Assume the first model $\f_1$ to be true, i.e.\ having fixed parameters $\para_1$. Write $\f_1(x) = \f_1(x,\para_1)$. The alternative model $\f_2$ stays parametrized with parameters $\para_2\in\Para_2$. Then, the \textit{T-optimality criterion}, or short \textit{T-criterion}, for model discrimination is defined as
	\begin{alignat*}{2}
		T(\xi) &:= \min_{\para_2\in\Para_2} \int_{\X} \left[\f_1(x) - \f_2(x,\para_2)\right]^2 \xi(dx)\\
		&~=\min_{\para_2\in\Para_2}\, \sum_{i=1}^n \w_i \left[\f_1(x_i) - \f_2(x_i,\para_2)\right]^2
	\end{alignat*}
	for a design $\xi = \left\lbrace (x_i, \w_i)\right\rbrace_{i=1,\ldots,n}$. A design $\xi^*$ is called \textit{T-optimal} if 
	\begin{align}
		\label{def:Topt}
		\xi^* = \argmax_{\xi\in\Xi} ~T(\xi) = \argmax_{\xi\in\Xi} \min_{\para_2\in\Para_2} \int_{\X} \left[\f_1(x) - \f_2(x,\para_2)\right]^2 \, \xi(dx).
	\end{align}
\end{defi}


One benefit of the T-criterion is its relative independence from data. In fact, if we set the true parameters $\para_1$ of the first model based on prior assumptions, then we can even compute a T-optimal design without any data at all. Nevertheless, if data $\mathcal{D}_N = \{(x_i, y_i)\mid i=1,\ldots,N\}$ is available, it is still best practice to use this data in order to estimate $\para_1$, e.g.\ by the method of least-squares.

To work with the T-criterion, we introduce the following notations.

\begin{nota}
	We define the squared model difference as
	\[\varphi(x,\para_2) := \left[\f_1(x) - \f_2(x,\para_2)\right]^2.\]
	We further define the function
	\[T(\xi,\para_2) := \int_{\X} \varphi(x, \para_2) ~\xi(dx),\]
	and the optimal parameter with respect to a given design $\xi$ as
	\[\hat{\para}_2 := \hat{\para}_2(\xi) := \argmin_{\para_2\in\Para_2} ~T(\xi,\para_2).\]
	With these notations, we may also write the T-criterion as
	$T(\xi) = T(\xi, \hat{\para}_2).$
\end{nota}

One issue is that the optimal parameters $\hat{\para}_2$ are not necessarily unique. This cannot be avoided in general, e.g.\ for designs with a single support point. 
Nevertheless, we are mostly interested in designs that contain several design points. And for those, in practice, one would assume that they have in fact a unique solution to the corresponding least-squares problem. For example, linear models have a unique solution to the least-squares problem if the number of support points equals the number of parameters. This justifies the notion of regular designs.

\begin{defi}(Regular Design)
	A design $\xi$ is called \textit{regular} if the corresponding weighted least-squares problem $\min_{\para_2\in\Para_2} T(\xi,\para_2)$ has a unique minimizer $\hat{\para}_2(\xi)$.
\end{defi}

We further use the following assumptions on the model functions such that we can guarantee the existence of T-optimal designs and other nice properties like concavity (see e.g.\ \cite{FedDoe2013}).

\begin{assum}
	\label{ass:model}
	We assume the following properties on a single-response model $\f:\X\times\,\Para \to \mathbb{R}$:
	\begin{itemize}
		\item[(M1)] $\X\neq\emptyset$ is compact.
		\item[(M2)] $f$ is continuous with respect to $(x,\para) \in \X\times\,\Para$.
		\item[(M3)] $\Theta_2\neq\emptyset$ is compact and convex.
	\end{itemize}
\end{assum}


Another important concept when optimizing the T-criterion are directional derivatives. But beforehand, we need the following notation for designs with a single support point.

\begin{nota}
	We denote the probability measure with full weight on one design point $x\in\X$ as $\xi_x$, i.e.\ $\xi_x := \{(x,1)\}$.
\end{nota}

\begin{prop}[Directional Derivative of the T-Criterion]
	\label{prop:dirDerTcrit}
	The right-hand directional derivative for the T-criterion for some regular design $\xi$ in direction $\bar{\xi}-\xi$ exists and is given by
	\[\delta^+ \, T(\xi,\bar{\xi}-\xi) = \int_{\X}\varphi(x,\hat{\para}_2(\xi))~\bar{\xi}(dx) - T(\xi,\hat{\para}_2(\xi)).\]
	In particular, it holds
	\begin{align*}
		&T\big((1-\alpha)\xi + \alpha\bar{\xi}, \; \hat{\para}_2((1-\alpha)\xi + \alpha\bar{\xi})\big) \\
		&\qquad\qquad = T(\xi, \hat{\para}_2(\xi)) + \alpha \int_{\X}\psi(x,\xi;\hat{\para}_2(\xi)) ~\bar{\xi}(dx) + \BigO(\alpha^2),
	\end{align*}
	where the (right-hand) directional derivative in direction $\xi_x-\xi$ for any design point $x\in\X$ is given by
	\[\psi(x,\xi) := \psi(x,\xi;\hat{\para}_2(\xi)) := \delta^+ \, T(\xi,\xi_x-\xi) = \varphi(x,\hat{\para}_2(\xi)) - T(\xi,\hat{\para}_2(\xi)).\]
\end{prop}

The existence of the directional derivative is derived in, e.g., \cite{kuczewski2006diss}.

We are now able to state the Equivalence Theorem for the T-criterion, which gives a necessary and sufficient criterion for T-optimal solutions.

\begin{theorem}[Equivalence Theorem for the T-Criterion]
	\label{thm:equThm}
	The following statements hold:
	\begin{enumerate}
		\item A design $\xi^*$ is T-optimal if and only if there exists a measure $\zeta$ on the parameters such that
		\begin{equation}
			\label{expr:Tcrit:EqThm}
			\int_{\hat{\Para}_2(\xi^*)} \varphi(x, \hat{\para}_2) ~\zeta(d\hat{\para}_2) \leq T(\xi^*), ~\forall x\in\X,
		\end{equation}
		where $\hat{\Para}_2(\xi) := \argmin_{\para_2\in\Para_2} \int_{\X} \varphi(x,\para_2) ~\xi(dx)$ is the set of optimal parameters w.r.t.\ $\xi$, and $\zeta$ is a probability measure on $\hat{\Para}_2(\xi)$, i.e.\
		$\zeta\big(\hat{\Para}_2(\xi)\big) = 1.$
		\item The inequality \eqref{expr:Tcrit:EqThm} holds with equality for all support points of a T-optimal design $\xi^*$.	
		\item The set of optimal designs is convex.
	\end{enumerate}
	\label{thmGlobOptCond}
\end{theorem}

\begin{proof}
	See, for example, \cite{kuczewski2006diss}.
\end{proof}

We end this section with a helpful result concerning the existence of T-optimal solutions with finitely many support points.

\begin{prop}[Existence of T-optimal solutions with finite support]
	\label{thm:nrSupp}
	There always exists a T-optimal design $\xi^*$ with at most $\dim(\Theta_2)+1$ many design points.
\end{prop} 
\begin{proof}
	The statement can be shown by applying the characterization theorem for best uniform approximation, see \cite{braess2013optimal}.
\end{proof}

%% file: 2-ADAPT-MD/The_2-ADAPT-MD_Algorithm.tex
\section{The 2-ADAPT-MD Algorithm}
\label{sec:alg}

In the previous section we have seen that optimizing the T-criterion can be formulated as a maximin problem of the form
\begin{equation}
	\label{expr:originalSIP}
	\max_{\xi\in\Xi} \min_{\para_2\in\Para_2} \; \sum_{i=1}^{\dim(\Theta_2)+1} \w_i \left[\f_1(x_i) - \f_2(x_i,\para_2)\right]^2,
\end{equation}
where $\xi$ is a discrete probability measure with $\dim(\Theta_2)+1$ many support points. Thus, it can also be formulated as a semi-infinite problem and solved by the Blankenship \& Falk algorithm (see Appendix \ref{App:SIP}). Similar formulations were also used be Kuczewski \cite{kuczewski2006diss} and Duarte et al.\ \cite{duarte2015semi}. However, the upper-level problem, which has to be solved in every iteration, is 
\begin{align*}
	\label{expr:upperLevelTSIP}
	\max_{\substack{t\in\mathbb{R}\\\w\in\mathbb{R}^N\\x_1,\ldots,x_N \in\X}}\quad &t \nonumber\\
	\text{ s.t.} \quad &\sum_{i=1}^{N} \w_i \left[\f_1(x_i) - \f_2(x_i,\para_2)\right]^2\geq t, \quad \forall\, \para_2\in\dot{\Para}_2,\\
	&\sum_{i=1}^{N} \w_i = 1,\quad \w\geq 0 \nonumber,
\end{align*}
where $N = \dim(\Theta_2)+1$ is the maximum number of required support points. Not only does this problem contain many variables, $(d_x+1)\cdot N + 1$ to be exact, but it is also often highly non-linear in the design points, which makes it hard to solve.

Therefore, we present a new algorithm for computing T-optimal designs, the \textproc{2-ADAPT-MD}. Although we do not eradicate the problems with the SIP approach completely, we mitigate the issues by exploiting the linearity in the weights, and optimizing with respect to the design points as infrequently as possible. 

The idea for the \textproc{2-ADAPT-MD} algorithm is the same as Algorithm 3.2 by Dette et al.\ \cite{dette2015bayesian}: We use a sub-routine to compute T-optimal designs on a finite subset of design points, then we augment this subset by a new design point to improve the solution with regard to the entire set of design points. 

But unlike the approach by Dette et al., our new method does not rely on a linearization. This is advantageous for cases where the linearization is not good enough of an approximation. Also, our algorithm is able to deal with multi-response models, whereas the linearization formulated by Dette et al.\ is limited to single-response models.


\input{2-ADAPT-MD/DISC-MD_Sub-Routine}

\input{2-ADAPT-MD/2-ADAPT-MD_algorithm}

%% file: 2-ADAPT-MD/DISC-MD_Sub-Routine.tex
\subsection{The Grid-Based \textproc{DISC-MD} Sub-Routine}

At first we consider the sub-problem to find T-optimal designs on a finite set of design points ${\X_N=\{x_1,\ldots,x_N\} \subset \X}$. This is:
\begin{equation}
	\label{expr:maximin:Tkrit}
	\max_{\substack{{\w\geq0}\\ \Vert\w\Vert_1=1}} \min_{\para_2\in\Para_2} \; \sum_{i=1}^N \w_i \left[\f_1(x_i) - \f_2(x_i,\para_2)\right]^2.
\end{equation}
We again reformulate this maximin problem as a SIP. But this time we get a \textit{linear} SIP:
\begin{align}
	\label{expr:upperLevelSIP}
	\max_{\substack{t\in\mathbb{R}\\\w\in\mathbb{R}^N}}\quad &t \nonumber\\
	\text{ s.t.} \quad &\sum_{i=1}^{N} \w_i \left[\f_1(x_i) - \f_2(x_i,\para_2)\right]^2\geq t, ~\forall \para_2\in\Para_2,\\
	&\sum_{i=1}^{N} \w_i = 1,\quad \w\geq 0 \nonumber.
\end{align}
Notice that in this formulation, we only optimize with respect to the weights (and $t$) since the former inner optimization problem is replaced by infinitely many constraints.

To solve (\ref{expr:upperLevelSIP}), we may use the Blankenship \& Falk algorithm. So we approximate the SIP by considering finitely many constraints, i.e.\ we consider a finite discretization of the parameter set $\dot{\Para}_2 \subset \Para_2$. This yields the upper-level problem
\begin{align}
	\label{expr:upperLevelLP}
	\max_{\substack{\w\in\mathbb{R}^N\\t\in\mathbb{R}}}\quad &t \nonumber\\
	\text{ s.t.} \quad &\sum_{i=1}^{N} \w_i \left[\f_1(x_i) - \f_2(x_i,\para_2)\right]^2\geq t, \quad \forall\, \para_2\in\dot{\Para}_2,\\
	&\sum_{i=1}^{N} \w_i = 1,\quad \w\geq 0 \nonumber,
\end{align}
which is a linear program.
The corresponding lower-level problem aims to find new parameters for the discretization $\dot{\Para}_2$. This is the weighted least-squares problem
\begin{equation}
	\label{expr:WLSforDISC_MD}
	\min_{\para_2\in\Para_2} \; \sum_{i=1}^N w_i  \left[\f_1(x_i) - \f_2(x_i,\para_2)\right]^2,
\end{equation}
where the weights are given by the solution of the upper-level problem (\ref{expr:upperLevelLP}). 

The complete approach is presented in Algorithm \ref{alg:fixGrid}, which we denote by \textproc{DISC-MD}. The name stems from the fact that we assume a discrete (and finite) set $\X_N$ for the set of all design points. 


\begin{algorithm}[thb]
	\caption{Gird-Based \textproc{DISC-MD} Algorithm}
	\label{alg:fixGrid}
	\begin{algorithmic}[1]\doublespacing
		\Require{$s=0$; finitely many design points $\X_N = \{x_1,\ldots,x_N\}$; initial design $\xi\high{0}$; initial discretization $\dot{\Para}_2\high{0}\subset\Para_2$}
		\Procedure{Disc-MD}{$\X_N$, $\xi\high{0}$, $\dot{\Para}_2\high{0}$}
		\Repeat 
		\State obtain $\w\high{s+1}$ by solving (\ref{expr:upperLevelLP}) given discretization $\dot{\Para}_2\high{s}$
		\Statex ~ \Comment{compute new weights (upper-level)}
		\State obtain $\hat{\para}_2\high{s+1}$ by solving (\ref{expr:WLSforDISC_MD}) given weights $\w\high{s+1}$
		\State $\dot{\Para}_2\high{s+1} = \dot{\Para}_2\high{s} \cup \left\{\hat{\para}_2\high{s+1}\right\}$ \Comment{update parameter discretization (lower-level)}
		\State $s := s+1$ \Comment{iterate}
		\Until{convergence criterion is met}\\
		\Return $\hat{\xi} = \left\{\left(x_i, w_i\high{s}\right) \big\vert~ i=1,\ldots,N\right\}$, $\dot{\Para}_2\high{s}$
		\EndProcedure
	\end{algorithmic}
\end{algorithm}

%% file: 2-ADAPT-MD/2-ADAPT-MD_algorithm.tex
\subsection{The 2-ADAPT-MD Algorithm}

Now we want to find T-optimal designs on a continuous set of design points $\X$. So we want to solve
\begin{equation}
	\label{expr:origSIP}
	\xi^* = \argmax_{\xi\in\Xi(\X)} \; \min_{\para_2\in\Para_2} \int_{\X} \varphi(x, \para_2) \, \xi(dx).
\end{equation}
To do so, we approximate the design set $\X$ by a set of finitely many design points $\X_N$ and compute an optimal solution $\xi^*_N$ on this finite set. This is done by the \textproc{DISC-MD} method which yields
\begin{equation}
	\label{expr:discTcrit}
	\xi^*_N = \argmax_{\xi\in\Xi(\X_N)} \; \min_{\para_2\in\Para_2} \, \sum_{i=1}^N \, \xi(x_i) \, \varphi(x_i, \para_2)
\end{equation}
in the limit. Obviously $T(\xi^*_N) \leq T(\xi^*)$, but for good choices for the design points in $\X_N$ we also get a good approximate solution $\xi_N^*$.

In fact, recall there always exists an optimal design $\xi^*$ with at most $\dim(\Para_2) + 1$ many design points (Theorem \ref{thm:nrSupp}). Finding these few design points such that they are included in the finite set $\X_N$ is sufficient for $T(\xi^*_N) = T(\xi^*)$.

The question is now: How do we find new design points to get better designs? Like with the Vector Direction Method and the algorithm by Dette et al.\ \cite{dette2015bayesian}, we take the design point that maximizes the directional derivative from Proposition \ref{prop:dirDerTcrit}, i.e.\
\begin{equation}
	\label{expr:newPoint}
	x\high{N+1} = \argmax_{x\in\X}\, \psi(x,\xi^*_N) = \argmax_{x\in \X}\, \varphi(x,\hat{\para}_2(\xi^*_N)).
\end{equation}
This point is of interest since it induces a direction of ascent. It is also the point with highest potential to increase the objective $T(\cdot)$ as $\varphi(x,\para_2)$ is an upper bound for the optimal objective value $T(\xi^*)$.

Altogether, we end up with Algorithm \ref{alg:SIPNewApp}, which is called \textproc{2-ADAPT-MD}. The name follows from the idea that we consider \textit{two} discretizations which we \textit{adaptively} improve to compute T-optimal designs for \textit{model discrimination}. On the one hand, we augment the discretization of the parameter set $\dot{\Para}_2$ to get better approximate designs $\hat{\xi}_N$ to the T-optimal design $\xi^*_N$ on $\X_N$. On the other, we determine new points to update $\X_N$ so that $\xi^*_N$ becomes a better approximation to the optimal design $\xi^*$ on $\X.$

\begin{algorithm}[htb]
	\caption{\textproc{2-ADAPT-MD} Algorithm for the T-Criterion}\label{alg:SIPNewApp}
	\begin{algorithmic}[1] \doublespacing
		\Require{finitely many design points $\X\high{0} = \{x_1,\ldots,x_N\}$; initial design $\xi\high{0}$, initial discretization $\dot{\Para}_2\high{0}\subset\Para_2$}
		\Repeat
		\State $\xi\high{n+1}, ~\dot{\Para}_2\high{n+1} = \textsc{Disc-MD}(\X\high{n}, \xi\high{n}, \dot{\Para}_2\high{n})$ \Comment{compute design}
		\State obtain $\hat{\para}_2\high{n+1}$ by solving the LS problem (\ref{expr:WLSforDISC_MD}) given design $\xi\high{n+1}$
		\State $\dot{\Para}_2\high{n+1} := \dot{\Para}_2\high{n+1} \cup \left\{\hat{\para}_2\high{n+1}\right\}$
		\State obtain $x\high{n+1}$ by maximizing the squared distance (\ref{expr:newPoint}) given parameters $\hat{\para}_2\high{n+1}$
		\State $\X\high{n+1} :=
		\X\high{n} \cup \left\{x\high{n+1}\right\}$ \Comment{augment candidate design points}
		\State $n := n+1$ \Comment{iterate}
		\Until{convergence criterion is met}\\
		\Return $\xi\high{n}$
	\end{algorithmic}
\end{algorithm}

\begin{bem}[\textproc{2-ADAPT-MD} for Multi-Response Models]
	\label{bem:novel:2ADA}
	With the \textproc{2-ADAPT-MD} algorithm, we can also compute T-optimal designs for multi-response models $f:\X\times\Para \to \Y$ with $\Y \subseteq \mathbb{R}^{d_y}$ and $d_y > 1$. This is achieved by changing the squared distance $\varphi$ to
	\[\varphi(x,\para_2) = \left\Vert f_1(x) - f_2(x,\para_2) \right\Vert^2.\]
\end{bem}

%% file: Convergence/Convergence.tex
\section{Convergence}

\label{sec:conv}

We now show that we can find arbitrary good approximations to T-optimal designs with the \textproc{2-ADAPT-MD} algorithm.

\input{Convergence/Convergence_of_DISC-MD}

\input{Convergence/Convergence_of_2-ADAPT-MD}

%% file: Convergence/Convergence_of_DISC-MD.tex
\subsection{Convergence of the DISC-MD Procedure}

Our initial focus is on the \textproc{DISC-MD} sub-routine. We show that we find arbitrary good approximations of the T-optimal design $\xi_N^*$ on the finite design space $\X_N$ by using again results from semi-infinite programming.

\begin{theorem}[Convergence of the DISC-MD Procedure]
	\label{conv_disc_MD}
	Let $\dot{\Para}_2\high{0}$ be a finite discretization of $\Para_2$. Assume that the design $\xi\high{s}$ is regular in each iteration, i.e.\ we find unique minimizers in the least-squares problem in the lower-level. In particular, the initial design $\xi\high{0}$ is regular. Then, every accumulation point of the sequence of solutions $(\xi\high{s})_{s\in\mathbb{N}}$ of Algorithm \ref{alg:fixGrid} is a T-optimal solution on $\X_N$.
\end{theorem}
\begin{proof}
	We prove the statement by showing that we satisfy the conditions of Theorem \ref{thm:convBF} (see Appendix \ref{App:SIP}). If its requirements are satisfied, every accumulation point of the sequence of solutions is an optimal solution of the semi-infinite problem (\ref{expr:upperLevelSIP}). Hence, it is a T-optimal solution on $\X_N$.	
	
	So, we have to show that the feasible region of the initial iteration
	\begin{align*}
		\mathcal{F}(\dot{\Para}_2\high{0}) = \Big\{(w,t)\in\mathbb{R}^{N+1} \mid ~&t\leq \sum_{i=1}^N w_i \left[\f_1(x_i) - \f_2(x_i,\para_2) \right]^2, ~\forall \, \para_2\in\dot{\Para}_2\high{0},\\
		&\sum_{i=1}^N w_i = 1, \quad w_i\geq 0\Big\}
	\end{align*}	
	is compact. Since $\mathcal{F}(\dot{\Para}_2\high{0})$ can be defined by finitely many linear inequalities, it is a closed convex polytope. It remains to show that $\mathcal{F}(\dot{\Para}_2\high{0})$ is bounded. The constraints imply 
	${w\in[0,1]^N}$,
	and 
	$$t\leq \max_{\para\in\dot{\Para}_2\high{0}} \, \sum_{i=1}^N \left[\f_1(x_i) - \f_2(x_i,\para_2) \right]^2 := C < \infty.$$
	W.l.o.g., we can further assume
	$t\geq 0.$
	This is due to maximizing with respect to $t$ and all upper bounds of $t$ being positive. Altogether, $\mathcal{F}(\dot{\Para}_2\high{0})$ is bounded and closed in $\mathbb{R}^{N+1}$, i.e.\ compact.
\end{proof}

%% file: Convergence/Convergence_of_2-ADAPT-MD.tex
\subsection{Convergence of the 2-ADAPT-MD Algorithm}

We have already highlighted that the 2-ADAPT-MD algorithm is reminiscent of the algorithm by Dette et al.\ \cite{dette2015bayesian}. In fact, if the \textproc{DISC-MD} routine returned an optimal design on $\X_N$, then Theorem 3.3 from their paper proves convergence. But for this sub-routine, we only have convergence in the limit. Therefore, in finite time, we are only guaranteed that $\hat{\xi}_N$ is an arbitrary good approximation to $\xi^*_N$. Hence, we introduce the notion of $\varepsilon$-T-optimal designs.

\begin{defi}[$\varepsilon$-T-Optimal Designs]
	\label{def:epsTopt}
	A design $\xi$ is called \textit{$\varepsilon$-T-optimal} if
	\[ \max_{x\in\X} \, \psi(x,\xi) \leq \varepsilon.\]
\end{defi}

Definition \ref{def:epsTopt} means that the directional derivative $\psi(x,\xi) = \varphi(x,\hat{\para}_2(\xi)) - T(\xi)$ is small. We also conclude that if we have an $\varepsilon$-T-optimal solution $\xi$, then the objective value of this design $T(\xi)$ is close to the optimal objective value. This justifies our definition for approximate solutions.

\begin{cor}
	If $\xi$ is an $\varepsilon$-T-optimal design and $\xi^*$ is an actual T-optimal design on the same design and parameter spaces, then
	\[T(\xi) \geq T(\xi^*) - \varepsilon.\]
\end{cor}
\begin{proof}
	By the duality of the T-criterion and the Chebyshev approximation problem (see \cite{dette2009titoff}), we have
	\[\max_{x\in\X} \, \varphi(x,\hat{\para}_2(\xi)) \geq \min_{\para_2\in\Para_2} \, \max_{x\in\X} \varphi(x,\para_2) = T(\xi^*).\]
	Therefore,
	\[T(\xi) \geq \max_{x\in\X} \, \varphi(x,\hat{\para}_2(\xi)) - \varepsilon \geq T(\xi^*) - \varepsilon\]
	holds.
\end{proof}

Additionally, we will need the following lemma.

\begin{lem}
	\label{lem:existsOfSol}
	Let Assumptions (M1)-(M3) hold. Let $(\xi_n)_{n\in\mathbb{N}}$ be a sequence of designs such that the objective values converge, i.e.\ $T(\xi_n) \to T^*$ as $n\to\infty$ for some $T^*\in\mathbb{R}$. Then, there exists a solution $\xi^* \in\Xi(\X)$ such that
	\[T(\xi^*) = T^*.\]
\end{lem}
\begin{proof}
	The sequence of solutions $(\xi_n)_{n\in\mathbb{N}}$ is obviously tight because the solutions are probability measures defined on the compact set $\X$. But then, by Prokhorov's Theorem (see e.g. \cite{parthasarathy2005probability}), there exists a subsequence $(\xi_{n_k})_{k\in\mathbb{N}}$ which weakly converges to a probability measure $\xi^*$ defined on $\X$. And by the weak continuity of the T-criterion \cite{kuczewski2006diss}, we have
	\[\lim\limits_{k\to\infty} T(\xi_{n_k}) = T(\xi^*).\]
	But since the objectives $(T(\xi_n))_{n\in\mathbb{N}}$ converges to $T^*$ by assumption, we also have
	\[\lim\limits_{k\to\infty} T(\xi_{n_k}) = T^*.\]
	So by the uniqueness of limits, we have $T(\xi^*) = T^*$.
\end{proof}

We now show that the new algorithm finds arbitrary good approximations of the T-optimal design $\xi^*$ as long as the solution of the \textproc{DISC-MD} sub-routine returns a sufficiently good approximation for the optimal design on a finite subset $\X_N$.

But before we can prove the statement, we have to introduce a bound for the second directional derivative. Recall from Proposition \ref{prop:dirDerTcrit} that for a design $\xi$, and the directional derivative $\psi(x,\xi)$ in direction $\xi_x-\xi$, we have
\begin{align}
	\label{expr:K}
	&T\big((1-\alpha)\xi + \alpha\bar{\xi}, \, \hat{\para}((1-\alpha)\xi + \alpha\bar{\xi})\big) \nonumber \\
	&\qquad \qquad = T(\xi, \hat{\para}(\xi)) + \alpha \int_{\X}\psi(x,\xi;\hat{\para}(\xi)) ~\bar{\xi}(dx) + \BigO(\alpha^2) \\
	&\qquad \qquad \leq T(\xi, \hat{\para}(\xi)) + \alpha \int_{\X}\psi(x,\xi;\hat{\para}(\xi)) ~\bar{\xi}(dx) + \alpha^2 K \nonumber
\end{align}
for some $K\in\mathbb{R}$ which depends on the given models. This $K$ can be understood as a bound for the second directional derivative, but also plays a role on how close we can get to a T-optimal solution.

\begin{theorem}[Convergence of SIP-Based Algorithm for T-Criterion]
	\label{conv_SIP_new_alg}
	Let $\varepsilon>0$. Assume in each iteration $n$ that \textproc{DISC-MD} returns a regular $\varepsilon$-T-optimal design $\xi\high{n}$ on the finite design space $\X\high{n}\subset\X$. Then, a subsequence of intermediate solutions $(\xi\high{n_j})_{j\in\mathbb{N}}$ of Algorithm \ref{alg:SIPNewApp} converges to a  $2\sqrt{K\varepsilon}$-T-optimal design $\hat{\xi}$ on $\X$, i.e.
	\[\max_{x\in\X} \, \psi(x, \hat{\xi}) \leq 2\sqrt{K\varepsilon}\]
	with $\xi\high{n_j} \to \hat{\xi}$ as $j\to\infty$.
\end{theorem}
\begin{proof}
	Let $\hat{\xi}_N$ be the solution found by the \textproc{DISC-MD} sub-routine for the finite design space $\X_N = \{x_1,\ldots,x_N\}$ as input, and let $\xi^*_N$ be an actual T-optimal design on $\X_N$. Since the sub-routine stops with an $\varepsilon$-T-optimal solution on $\X_N$, we have
	\begin{equation}
		\label{expr:proofConHatStar}
		T(\xi^*_N) \geq T(\hat{\xi}_N) \geq T(\xi^*_N)-\varepsilon.
	\end{equation}
	
	We divide this proof into two parts. First, we show that the sequence of solutions $(\hat{\xi}_N)_{N\in\mathbb{N}}$ converges weakly to a design $\hat{\xi}$. Then, we prove that $\hat{\xi}$ is a $2\sqrt{K\varepsilon}$-T-optimal solution.
	
	
	We start with the sequence of optimal objective values $(T(\xi^*_N))_{N\in\mathbb{N}}$. It is monotone increasing because we may only add points to the finite design space $\X_N$ in each iteration, thus increasing the feasible region $\Xi(\X_N)$ for the designs. Also, $(T(\xi^*_N))_{N\in\mathbb{N}}$ is bounded from above by $T(\xi^*)$, which is the objective value of the T-optimal design on the whole set $\X$. Therefore this sequence converges to an element $T^{**}\in\mathbb{R}$, i.e.\
	\begin{equation}
		\label{expr:proof:limitStar}
		\lim\limits_{N\to\infty} T(\xi_N^*) = T^{**}.
	\end{equation}
	Next, we show that the sequence of the $\varepsilon$-T-optimal solutions $(\hat{\xi}_{N})_{N\in\mathbb{N}}$ has a weakly convergent subsequence. Recall (\ref{expr:proofConHatStar}), which states
	\begin{equation}
		\label{expr:proof:hatInStar}
		T(\hat{\xi}_N) \in \left[T(\xi^*_N)-\varepsilon,\; T(\xi^*_N)\right].
	\end{equation}
	Combining (\ref{expr:proof:limitStar}) and (\ref{expr:proof:hatInStar}), we have that for a fixed $\delta>0$, there exists a sufficiently large $N_0\in\mathbb{N}$ such that for all $N>N_0$:
	\begin{align*}
		T^{**} - T(\hat{\xi}_N) &= T^{**} - T(\xi^*_N) + T(\xi^*_N) - T(\hat{\xi}_N)\\
		&\leq \delta + \varepsilon.
	\end{align*}
	As a consequence, the sequence $(T(\hat{\xi}_{N}))_{N\in\mathbb{N}}$ is bounded within a compact subset of $\mathbb{R}$. Specifically
	\[T(\hat{\xi}_N) \in \left[T^{**}-\delta-\varepsilon,\; T^{**}\right]\]
	for $N>N_0$. So there exists a convergent subsequence $(T(\hat{\xi}_{N_{j}}))_{j\in\mathbb{N}}$ with  
	\[\lim\limits_{j\to\infty} T(\hat{\xi}_{N_{j}}) = \hat{T}\]
	for some $\hat{T}\in\mathbb{R}$. And by the same arguments as above, we take another subsequence which weakly converges to a design $\hat{\xi}$ with \[\hat{T} = T(\hat{\xi}) \in \left[T^{**} - \varepsilon, T^{**}\right].\] For the sake of convenience we only denote the final subsequence by $(\hat{\xi}_s)_{s\in\mathbb{N}}$.
	
	
	We finally show that $\hat{\xi}$ is a $2\sqrt{K\varepsilon}$-T-optimal solution. Assume it is not, i.e.\
	\begin{equation}
		\label{expr:proof:assumptionHat}
		\max_{x\in\X} \psi(x, \hat{\xi}) = \max_{x\in\X} \, \varphi(x,\hat{\para}_2(\hat{\xi})) - T(\hat{\xi}) > 2\sqrt{K\varepsilon}.
	\end{equation}
	Due to continuity of the T-criterion w.r.t.\ weakly convergent designs, weak outer semi-continuity of the optimal parameters (see both in \cite{kuczewski2006diss}), and the obvious continuity of $\varphi$, we have that for any $\varepsilon'>0$, there is a sufficiently large $s_0\in\mathbb{N}$ such that for all $s\geq s_0$:
	\begin{equation}
		\label{expr:proof:modExpressionHat}
		\max_{x\in\X} \, \psi(x, \hat{\xi}_s) \geq 2\sqrt{K\varepsilon} + \varepsilon'.
	\end{equation}
	Next, define a design ${\xi}_{s+1}\high{\text{LS}}$ by
	\[\xi_{s+1}\high{\text{LS}} := (1-\hat{\alpha}_s)\hat{\xi}_s + \hat{\alpha}_s \, \xi_{\hat{x}_s},\]
	where
	\[\hat{x}_s = \argmax_{x\in\X} \, \psi(x, \hat{\xi}_s),\]
	and
	\[\hat{\alpha}_s = \argmax_{\alpha\in[0,1]} \; (1-\alpha)\hat{\xi}_s + \alpha \, \xi_{\hat{x}_s}.\]
	Basically, ${\xi}_{s+1}\high{\text{LS}}$ is the design which we get if we took $\hat{\xi}_s$ and applied optimal line search. In this case, $\hat{x}_s$ induces a direction of ascent. The defined design serves as a lower bound for our actual optimal design $\xi^*_{s+1}$ on the set $\X_{s+1} := \X_s \cup \{\hat{x}_s\}$, i.e.\
	\begin{equation}
		\label{expr:proof:optLargerLS}
		T(\xi^*_{s+1}) \geq T(\xi_{s+1}\high{\text{LS}}).
	\end{equation}	
	Because $\xi_{s+1}\high{\text{LS}}$ is a convex combination of designs with an optimal step-length $\hat{\alpha}_s\in[0,1]$, we use (\ref{expr:K}), and our modified assumption (\ref{expr:proof:modExpressionHat}) in order to conclude
	\begin{align*}
		T(\xi\high{\text{LS}}_{s+1}) - T(\hat{\xi}_s) &= T(\hat{\xi}_s) + \hat{\alpha}_s  \psi(\hat{x}_s,\hat{\xi}_s) + \BigO(\hat{\alpha}_s^2) - T(\hat{\xi}_s)\\
		&= \hat{\alpha}_s \psi(\hat{x}_s,\hat{\xi}_s) + \BigO(\hat{\alpha}_s^2)\\
		&\geq \hat{\alpha}_s (\sqrt{2K\varepsilon} + \varepsilon') - \hat{\alpha}_s^2 K\\
		&\geq \frac{(2\sqrt{K\varepsilon} + \varepsilon')^2}{4K}\\
		&\geq \varepsilon + \frac{(\varepsilon')^2}{4K},
	\end{align*}
	where the second inequality from below follows for the sub-optimal choice of a step-length \[\alpha_s = \frac{2\sqrt{K\varepsilon} + \varepsilon'}{2K}.\]	
	Therefore, with (\ref{expr:proof:optLargerLS}), we get
	\begin{align*}
		T(\hat{\xi}_{s+1}) - T(\hat{\xi}_s) &\geq T(\xi^*_{s+1}) - \varepsilon - T(\hat{\xi}_s)\\
		&\geq T(\xi\high{\text{LS}}_{s+1}) - T(\hat{\xi}_s) - \varepsilon\\
		&\geq \frac{(\varepsilon')^2}{4K}.
	\end{align*}	
	But this yields
	\[T(\xi^*) - T(\hat{\xi}_0) \geq \sum_{s=0}^{n} T(\hat{\xi}_{s+1}) - T(\hat{\xi}_s) \geq \sum_{s=0}^{n} \frac{(\varepsilon')^2}{4K} \to \infty\]
	as $n\to\infty$, which is a contradiction to $T(\xi^*) < \infty$ and $T(\hat{\xi}_0) \geq 0$. So our assumption (\ref{expr:proof:assumptionHat}) is false and $\hat{\xi}$ must be a $2\sqrt{K\varepsilon}$-T-optimal solution.		
\end{proof}

%% file: NumericalResults/NumericalResults.tex
\section{Numerical Results}
\label{sec:num}

We tested our new algorithm on two examples from chemical process engineering. Additionally, we compared the performance against three other algorithms from the literature, namely the original Vector Direction Method used by Atkinson and Fedorov, the algorithm by Dette et al.\ \cite{dette2015bayesian}, and the original SIP formulation which solves the maximin problem (\ref{expr:originalSIP}) by the Blankenship \& Falk algorithm. All algorithms have been implemented in Python \cite{python}.

\input{NumericalResults/ImplementationDetails}

\input{NumericalResults/Examples}

%% file: NumericalResults/ImplementationDetails.tex
\subsection{Implementation Details}

In the following, we describe how we numerically solve the sub-problems that make up each algorithm.

\paragraph{Computing new parameters.}
We solve the least-squares problem (\ref{expr:WLSforDISC_MD}) with the help of either the \textit{scipy dogbox} solver, or the KNITRO solver, depending on the application. Also, in order to find a globally optimal solution, we do a multi-start where one starting point is the solution from the previous iteration, and all other starting points are determined by a Sobol sequence \cite{sobol1967distribution}.

However, we might run into cases where we have very few support points with positive weight. In a degenerate case it might happen that we have full weight on a single design point in an iteration, i.e.\ we get an intermediate solution with $\xi = \{(x_1,1),(x_2,0),\ldots,(x_N,0)\}$. This can happen due to numerical instabilities. In these cases, the least-squares problem does not have a unique optimal solution and may yield poor fits.
Thus, we use a regularization by adding a small weight to all currently selected design points $x_1,\ldots,x_N$. This yields the regularized least-squares problem
\begin{equation}
	\label{expr:num:LSreg}
	\min_{\para_2\in\Para_2} \: \sum_{i=1}^N \w_i \left[ \f_1(x_i) - \f_2(x_i\para_2)\right]^2 + \lambda \sum_{i=1}^N \left[ \f_1(x_i) - \f_2(x_i\para_2)\right]^2
\end{equation}
with some small \textit{least-squares regularizer} $\lambda>0$, e.g.\ $\lambda = 10^{-8}$. This problem often has a unique optimal solution, and it is still very close to the original problem. The idea to consider all design points $x_1,\ldots,x_N$ with an additional small weight can also be justified by the fact that the Chebyshev approximation problem is the dual of the T-criterion. So for the corresponding parameters of a T-optimal design, we require that the squared distances are low everywhere. 

\paragraph{Computing new design points.} The sub-problem of finding a new design point (\ref{expr:newPoint}) is in general non-convex and solved using BARON, or by enumeration if the design space was discrete.

\paragraph{Updating the design.} Computing a new design by the linear program (\ref{expr:upperLevelLP}) is done with the MOSEK solver via the pyomo interface. Additionally, for other algorithms, we solve the quadratic program in Dette's algorithm with the KNITRO solver, and use again BARON for the simultaneous optimization of the weights and design points in the original SIP approach.

\paragraph{Stopping rule.} We use a stopping criterion based on the Equivalence Theorem \ref{thm:equThm}. For regular designs, it simplifies to 
\begin{equation}
	\label{expr:stopCrit}
	\max_{x\in\X} \varphi(x,\hat{\para}_2(\xi)) - T(\xi, \hat{\para}_2(\xi)) \leq \varepsilon,
\end{equation}
where the $\varepsilon$ is a pre-determined tolerance. However, if the estimated parameters $\hat{\para}_2(\xi)$ are not exact, it might happen that this criterion is satisfied too early. Specifically, poorly estimated parameters might cause that the support points of the design $\xi$ are not all equal, which is also required by the Equivalence Theorem \ref{thm:equThm}. But then the maximal support point might still satisfy (\ref{expr:stopCrit}), whereas the others do not. Therefore, we propose a small adaption:
\begin{equation}
	\label{expr:stopCritAdapt}
	\min_{x\in\supp(\xi)} \varphi(x,\hat{\para}_2(\xi)) - T(\xi, \hat{\para}_2(\xi)) \leq \varepsilon.
\end{equation}
If (\ref{expr:stopCritAdapt}) is satisfied, so is (\ref{expr:stopCrit}), but it does recognize the aforementioned case.

%% file: NumericalResults/Examples.tex
\subsection{Examples}

In the following, we compare the performance of the Vector Direction Method (VDM), Dette's algorithm, the original SIP approach, and the new \textproc{2-ADAPT-MD} algorithm. We consider two examples: One simple example that should be easy to solve for all models, and a harder example based on a small system of ordinary differential equations. 

In Table \ref{tab:num:params}, we present the parameters of the algorithms. They can be tweaked to find a better solution or terminate earlier. If not otherwise stated, we take the default values. However, for the VDM, we make two distinctions:
\begin{itemize}
	\item The least-squares regularizer $\lambda$ is not required and thus set to zero since all support points keep a positive weight.
	\item The maximum number of iterations $n_\text{maxiter}$ is increased to 1000.
\end{itemize}

The following results were computed on an Intel Xeon Gold 6248R processor with four available cores. 

\begin{table}[th]
	\renewcommand{\arraystretch}{1.5}
	\setlength{\tabcolsep}{10pt}
	\begin{tabularx}{\textwidth}{lXl} \toprule 
		Parameter & Description & Default \\ \midrule 
		$\varepsilon$ & tolerance of stopping criterion & $10^{-5}$ \\ 
		$n_\text{maxiter}$  & maximum number of iterations & $100$ \\ 
		$n_\para$ & number of starting solutions to solve the least-squares problem (\ref{expr:WLSforDISC_MD}) or (\ref{expr:num:LSreg}), excluding the optimal parameters from the previous iteration & $9$ \\ 
		$\lambda$ & regularizer for least-squares problem (\ref{expr:num:LSreg}) (not needed for Fedorov's algorithm) & $10^{-8}$ \\ 
		$\varepsilon_\text{SIP}$ & tolerance of stopping criterion used in \textproc{DISC-MD} & $10^{-5}$ \\ 
		$n_\text{maxiter.SIP}$  & maximum number of iterations of \textproc{DISC-MD} & $20$ \\ 
		\bottomrule
	\end{tabularx} 
	\caption{Table of customizable parameters.}
	\label{tab:num:params}
\end{table}

\subsubsection*{Michaelis-Menten Model}

In our first example we consider the \textit{Michaelis-Menten (MM)} model. It gives a formula for the rate of a reaction where a substrate reacts with an enzyme to form a product:
\begin{equation*}
	f_2(x,(V,K)) = \frac{Vx}{K+x}.
\end{equation*}
Here, $x$ is a substrate concentration, $V$ is the maximum reaction rate, and $K$ is the so-called Michaelis-Menten constant which is the substrate concentration at which the reaction rate is at half maximum $\frac{1}{2} V$ \cite{bharath2021MM}. Alternatively, we might add a linear term to this rate. This yields the \textit{modified Michaelis-Menten (ModMM)} model:
\begin{equation*}
	f_1(x,(V,K,F)) = \frac{Vx}{K+x}+Fx.
\end{equation*}
Our goal is to find experiments such that we can discriminate between the two models and decide, which one fits our real-world problem better. This example as it stands was also studied in \cite{lopez2007optimal} and has since then become a standard example for model discrimination.

Considering that the original model is included in the modified version, we assume that $f_1$ is our true function with the parameters $V=1$, $K=1$, and $F=0.1$. Accordingly, $f_2$ is the alternative model. The design space is chosen to be $\X = [0.001, 5]$, and for the parameter space we have $(V,K) \in \Para_2 = [10^{-3}, 5]\times[10^{-3}, 5].$ Lastly, as the initial design we select
\[\xi\high{0} = \left\lbrace \begin{array}{cccc}
	1 & 2 & 3 & 4 \\ 
	\nicefrac{1}{4} & \nicefrac{1}{4} & \nicefrac{1}{4} & \nicefrac{1}{4}
\end{array}  \right\rbrace.\]

The performance of the algorithms for this example is presented in Table \ref{tab:num:MMvsMMM}. Additionally, the computed designs for each algorithm are shown in Table \ref{tab:num:MMvsMMM:optDes}.

\begin{table}[htp]
	\renewcommand{\arraystretch}{1.2}
	\setlength{\tabcolsep}{15pt}
	\begin{tabularx}{\textwidth}{lrrr} \toprule 
		Algorithm & \multicolumn{1}{c}{Reached Accuracy} & \multicolumn{1}{c}{$T(\xi^*)$}  & Run Time \\ \midrule 
		VDM & $9.64\cdot 10^{-6}$ & $1.1805\cdot 10^{-3}$ & 41.31 s \\
		Dette et al. & $5.02\cdot 10^{-6}$ & $1.1853\cdot 10^{-3}$ & 1.40 s \\
		original SIP & $4.75\cdot 10^{-6}$ & $1.1820\cdot 10^{-3}$ & 16511.88 s \\
		\textproc{2-ADAPT-MD} & $3.46\cdot 10^{-6}$ & $1.1854\cdot 10^{-3}$ & 6.91 s \\
		\bottomrule
	\end{tabularx} 
	\caption{Results for Michaelis-Menten vs.\ modified Michaelis-Menten.}
	\label{tab:num:MMvsMMM}
\end{table}

\begin{table}[htp]
	\renewcommand{\arraystretch}{1.5}
	\setlength{\tabcolsep}{10pt}
	\begin{tabularx}{\textwidth}{ll} \toprule
		Algorithm & Optimal Design $\xi^*$ \\ \midrule
		\vspace{0.3cm}
		VDM & $\left\lbrace \begin{array}{ccccccc}
			5 & 0.410 & 0.387 & 2.605 & 0.484 & 2.580 & \ldots \\
			0.2236 & 0.0062 & 0.0062 & 0.0062 & 0.0062 & 0.0062 & \ldots
		\end{array}  \right\rbrace$ \\ 
		\vspace{0.3cm}
		Dette & $\left\lbrace \begin{array}{ccc}
			0.387 & 2.581 & 5 \\ 
			0.3917 & 0.3897 & 0.2186\end{array} \right\rbrace$ \\
		\vspace{0.3cm}
		original SIP & $\left\lbrace \begin{array}{ccc}
			0.385 & 2.595 & 5 \\ 
			0.3904 & 0.3902 & 0.2194\end{array} \right\rbrace$ \\
		\textproc{2-ADAPT-MD} & $\left\lbrace \begin{array}{ccc}
			0.386 & 2.596 & 5 \\ 
			0.3906 & 0.3896 & 0.2198 \end{array} \right\rbrace$ \\
		\bottomrule
	\end{tabularx}
	\caption{Optimal designs per algorithm for Michaelis-Menten vs.\ modified Michaelis-Menten Model.}
	\label{tab:num:MMvsMMM:optDes}
\end{table}

Moreover, we illustrate the results. In Figure \ref{fig:num:MMvsMMM:fit}, we show how well the Michaelis-Menten model $f_2$ is fitted to the modified Michaelis-Menten model. Since all algorithms find designs with approximately the same optimal parameters
\begin{equation*}
	\hat{\para}_2 = (\hat{V}, \hat{K}) \approx (1.86, \, 2.15),
\end{equation*}
this figure is representative for all algorithms. Additionally in Figure \ref{fig:num:MMvsMMM:sensi}, we present a plot of the directional derivative $x\mapsto \psi(x,\xi^*)$ for the optimal design computed by the \textproc{2-ADAPT-MD} algorithm. It illustrates that the necessary and sufficient optimality criterion based on the Equivalence Theorem \ref{thm:equThm} is satisfied. We also plotted the directional derivative of the solution from the Vector Direction Method, which shows its problems to get rid of bad support points.

\begin{figure}[htp]
	\centering
	\includegraphics[width=.5\linewidth]{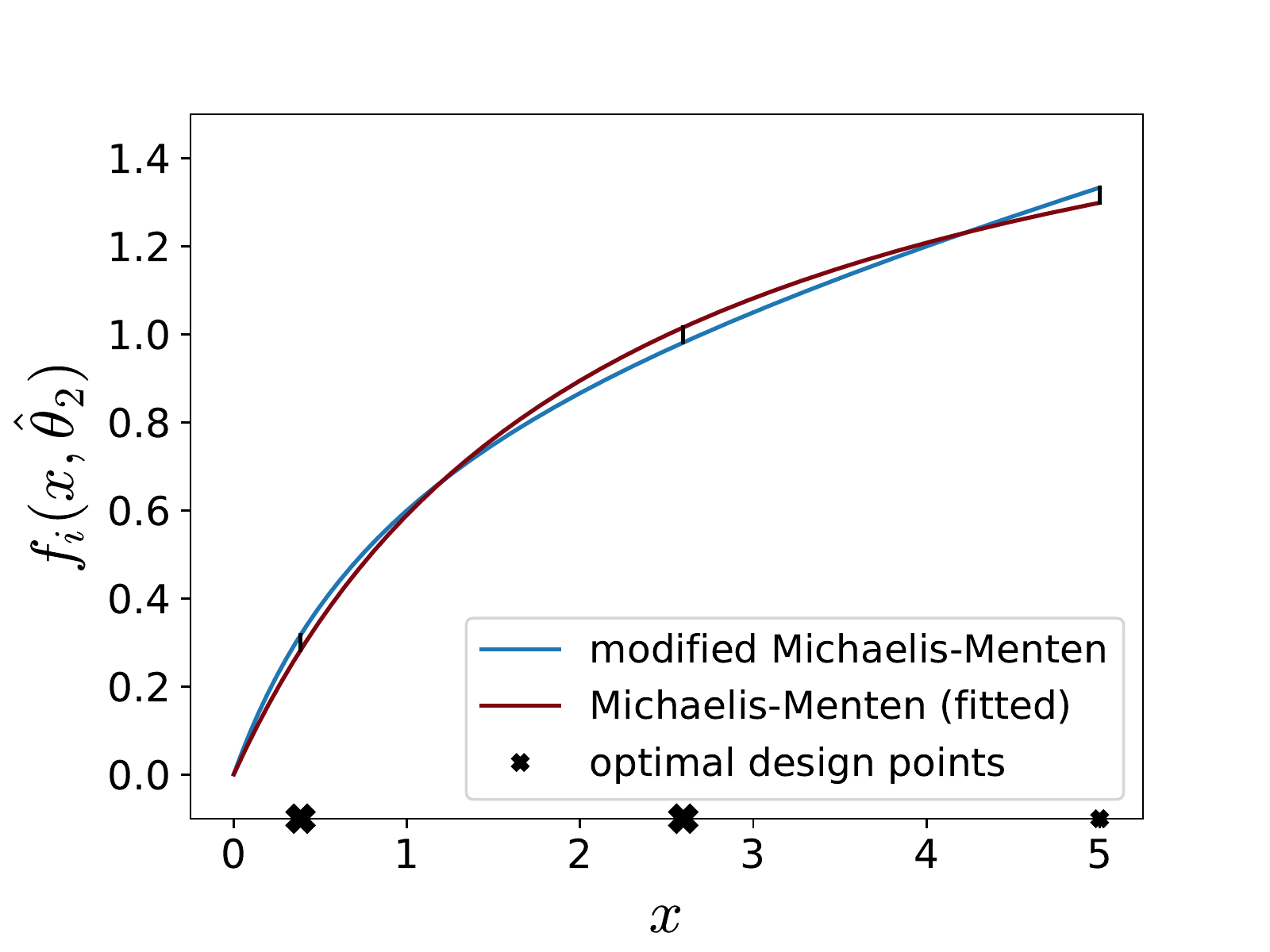}
	\caption{ModMM reference model vs.\ fitted MM alternative model with support points of optimal design from \textproc{2-ADAPT-MD} algorithm. The size of the crosses indicate the weights of the design, and the black lines highlight the model differences.}
	\label{fig:num:MMvsMMM:fit}
\end{figure}

\begin{figure}[htp]
	\centering
	\begin{subfigure}[t]{.44\textwidth}
		\centering
		\includegraphics[width=\linewidth]{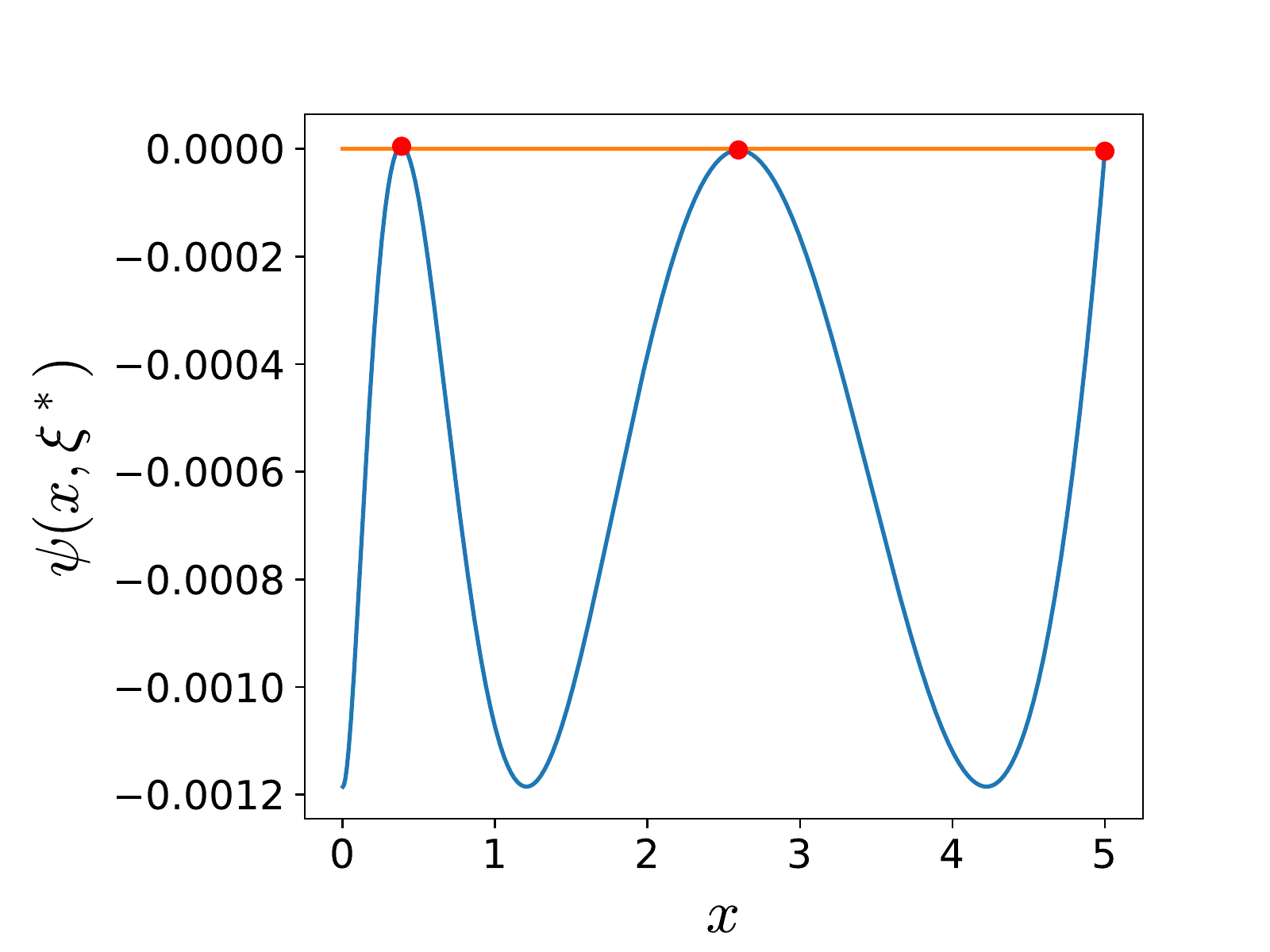}
		\caption{Computed by \textproc{2-ADAPT-MD}.}
	\end{subfigure}\hspace{.1\textwidth}
	\begin{subfigure}[t]{.44\textwidth}
		\centering
		\includegraphics[width=\linewidth]{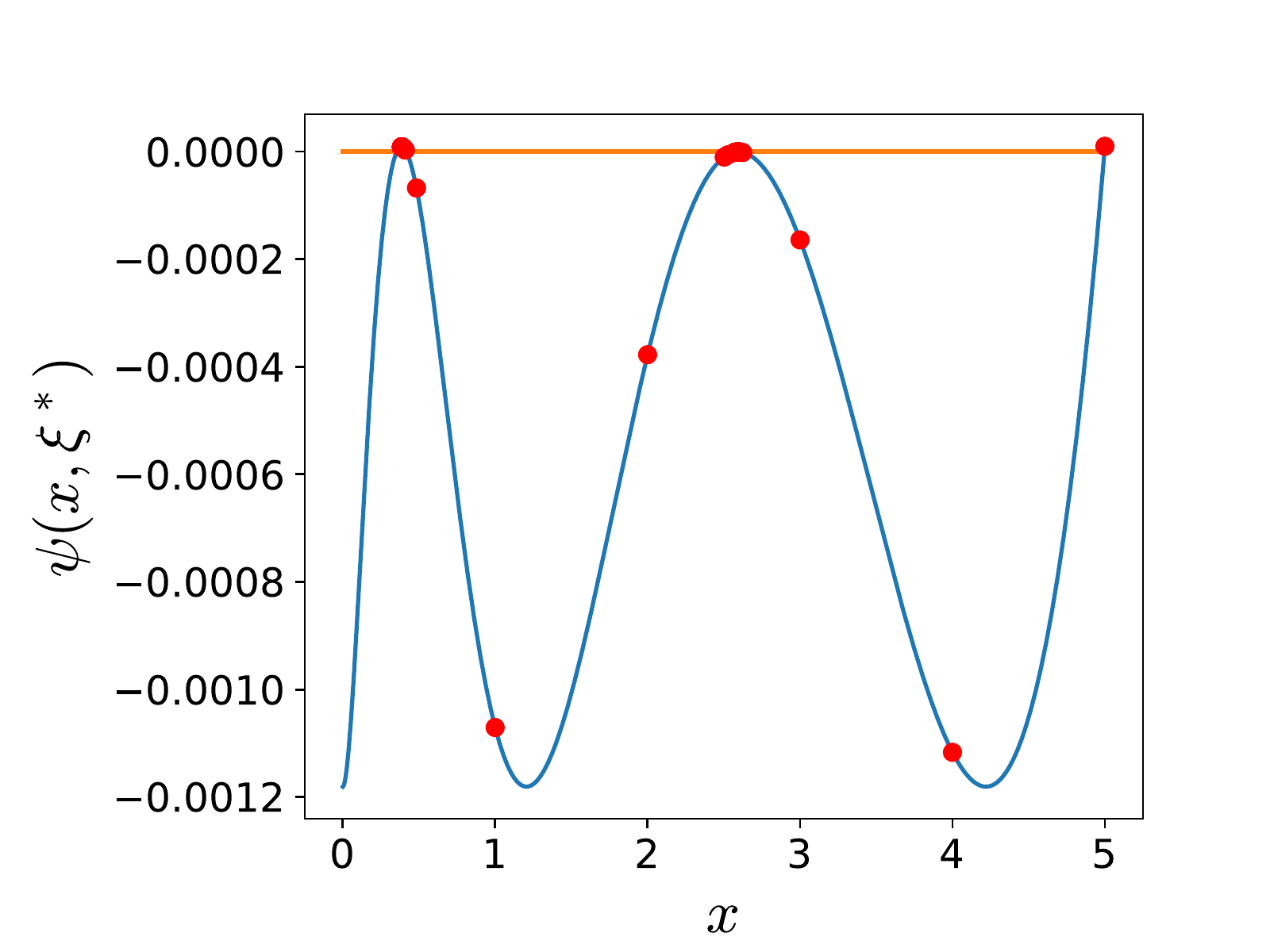}
		\caption{Copmuted by VDM.}
	\end{subfigure}
	\caption{Directional derivative (blue) with support points (red).}
	\label{fig:num:MMvsMMM:sensi}
\end{figure}

Overall, we see that all methods are capable of computing T-optimal designs. However, the algorithm by Dette et al.\ is the fastest in this easy example, closely followed by the \textproc{2-ADAPT-MD} algorithm. The Vector Direction Method already takes significantly longer, whereas the original SIP approach has tremendous problems solving the upper-level problem to global optimality, thus taking by far the longest.

\subsubsection*{Partially Reversible vs.\ Irreversible Consecutive Reaction}

Sometimes it is desirable to discriminate between an irreversible consecutive reaction of the form
\begin{equation}
	\label{expr:num:reacIrr}
	\ce{A ->[k_1] B ->[k_2] C},
\end{equation}
and a consecutive reaction where the first part is reversible, i.e.\
\begin{equation}
	\label{expr:reacRev}
	\ce{A <=>[k_1][k_3] B ->[k_2] C}.
\end{equation}
This problem was investigated in \cite{ucinski2005t}. Typically, reactions of such forms are modelled via a system of ordinary differential equations. For example, by modelling (\ref{expr:reacRev}) with power law expressions for the reaction rates, we get
\begingroup
\addtolength{\jot}{0.5em}
\begin{align}
	\frac{d[A]}{dt} &= -k_1 \, [A]^{n_1} + k_3 \, [B]^{n_3} \nonumber\\
	\frac{d[B]}{dt} &= k_1 [A]^{n_1} - k_2 [B]^{n_2} - k_3 [B]^{n_3} \\
	\frac{d[C]}{dt} &= k_2 [B]^{n_2}, \nonumber
\end{align}
\endgroup
where $[A]$, $[B]$, and $[C]$ are the concentrations of each component at a time $t$. Furthermore, $k_1$, $k_2$, and $k_3$ are rate constants, and $n_1$, $n_2$, and $n_3$ are reaction orders. Note that the model for the irreversible reaction (\ref{expr:num:reacIrr}) is included in this model by choosing $k_3 = 0$. And notice further that in this case, we deal with multi-response models that output the concentrations for all components $[A]$, $[B]$, and $[C]$ at time points $t$.

Now, if we want to discriminate between the models, we have to assume that the reaction with the reversible part (\ref{expr:reacRev}) is the reference model, and the irreversible reaction (\ref{expr:num:reacIrr}) is our alternative model. Otherwise, we could perfectly fit the alternative to the reference model, which means that the squared distance is zero everywhere and every design would be trivially optimal.

For the specific problem setting, we define the reference model by the parameter values $k_1 = 0.7$, $k_2=0.2$, $k_3=0.1$, $n_1=2$, $n_2=2$, and $n_3=1$. The parameter space for the alternative is given by $(k_1, k_2, n_1, n_2) \in \Para_2 := [0.5, 1.0]\times [0.05, 0.5]\times [1.5, 3.5]\times [1.5, 3]$.

As design variables, we consider the initial concentrations $[A]_0$, $[B]_0$, and $[C]_0$, and the time point of measurement $t$. However since evaluating the models at one point is already hard enough, we use a small discretized design space and maximize the squared distance (\ref{expr:newPoint}) by evaluating over the complete lattice instead of optimizing over the continuous space.
So we assume
$\left([A]_0, [B]_0, [C]_0, t\right) \in \X = \{0.5,0.7,0.9\}\times\{0.1,0.2,0.3\}\times\{0.0,0.15,0.3\}\times\{2,4,\ldots,10\}$.
The initial design is
\[\xi\high{0} = \left\{ \begin{array}{ccccc}
	\left(\begin{array}{c} 0.5\\0.1\\0\\2 \end{array}\right) & \left(\begin{array}{c} 0.5\\0.1\\0.15\\4 \end{array}\right) & \left(\begin{array}{c} 0.7\\0.3\\0.15\\6 \end{array}\right) & \left(\begin{array}{c} 0.9\\0.2\\0.15\\8 \end{array}\right)  & \left(\begin{array}{c} 0.9\\0.3\\0.3\\10 \end{array}\right) \vspace{0.5em} \\ 
	\nicefrac{1}{5} & \nicefrac{1}{5} & \nicefrac{1}{5} & \nicefrac{1}{5} & \nicefrac{1}{5}
\end{array}   \right\}.\] 

Due to this problem definition, the original SIP approach and the new \textproc{2-ADAPT-MD} algorithm both try to solve
\begin{equation*}
	\max_{\substack{{\w\geq0}\\ \Vert\w\Vert_1=1}} \min_{\para_2\in\Para_2} \; \sum_{i=1}^N \w_i \left[\f_1(x_i) - \f_2(x_i,\para_2)\right]^2.
\end{equation*}
But whereas the original SIP approach uses the Blankenship \& Falk algorithm directly, we iteratively determine a smaller subset of design points and apply the Blankenship \& Falk algorithm to this subset. This reduces the number of evaluations required to formulate the linear program (\ref{expr:upperLevelLP}) in the upper-level but should lead to more iterations overall.

The performance of the algorithms is presented in Table \ref{tab:num:backflow}. Though, for the original SIP method, we had to set the least-squares regularizer to $\lambda=0$, which worked in this example, but might lead to problems with other instances. Additionally, see Table \ref{tab:num:backflow:optDes} for the computed designs. Notice that this example discriminates between two multi-response models, and Dette's algorithm can only deal with single-response models, so its row is blank.

\begin{table}[htp]
	\renewcommand{\arraystretch}{1.2}
	\setlength{\tabcolsep}{15pt}
	\begin{tabularx}{\textwidth}{lrrr} \toprule 
		Algorithm & \multicolumn{1}{c}{Reached Accuracy} & \multicolumn{1}{c}{$T(\xi^*)$}  & Run Time \\ \midrule 
		VDM & $9.90\cdot 10^{-6}$ & $1.9275\cdot 10^{-3}$ & 2818.71 s \\
		Dette et al. & \multicolumn{1}{c}{-} & \multicolumn{1}{c}{-} & \multicolumn{1}{c}{-} \\
		original SIP & $4.04\cdot 10^{-6}$ & $1.9322\cdot 10^{-3}$ & 12042.51 s\\
		\textproc{2-ADAPT-MD} & $1.31\cdot 10^{-6}$ & $1.9322\cdot 10^{-3}$ & 345.30 s \\
		\bottomrule
	\end{tabularx} 
	\caption{Results for partially reversible vs.\ irreversible consecutive reaction model.}
	\label{tab:num:backflow}
\end{table}

\begin{table}[thp]
	\renewcommand{\arraystretch}{1.5}
	\setlength{\tabcolsep}{10pt}
	\begin{tabularx}{\textwidth}{ll} \toprule
		Algorithm & Optimal Design $\xi^*$ \\ \midrule
		\vspace{0.4cm}
		VDM & $\left\{ \begin{array}{cccc} (0.5,\, 0.1,\, 0 ,\, 2) & (0.9,\,  0.3, \, 0.3,\, 10) & (0.9, \, 0.1, \, 0.15 ,\, 10) & \ldots \\ 
			0.5563 & 0.3714 & 0.0336 & \ldots \end{array}   \right\}$ \\
		\vspace{0.4cm}
		original SIP & $\left\{ \begin{array}{cccc} (0.5,\, 0.1,\, 0 ,\, 2) & (0.9,\,  0.3, \, 0.3,\, 10) & (0.5, \, 0.1, \, 0 ,\, 10) \\ 
			0.5558 & 0.4120 & 0.0322 \end{array}   \right\}$ \\
		\vspace{0.4cm}
		\textproc{2-ADAPT-MD} & $\left\{ \begin{array}{ccc} (0.5,\, 0.1,\, 0 ,\, 2) & (0.9, \, 0.3, \, 0.3,\, 10) & (0.5, \, 0.1, \, 0 ,\, 10)  \\ 
			0.5562 & 0.4116 & 0.0322 \end{array}   \right\}$ \\
		\bottomrule
	\end{tabularx}
	\caption{Optimal designs per algorithm for partially reversible vs.\ irreversible consecutive reaction model.}
	\label{tab:num:backflow:optDes}
\end{table}

The results prove that we can not only find optimal designs even for multi-response models, but taking a small subset of candidate design points as it is done in the \textproc{2-ADAPT-MD} algorithm is indeed beneficial compared to the usual SIP approach. The additional effort to formulate the LP exceeds the cost of additional iterations. Also, the least-squares regularization can be applied more effectively, which can be crucial for instances where intermediate solutions with too few support points appear.

%% file: Conclussion/Conclussion.tex
\section{Conclusion}
\label{sec:conc}

In this paper we have introduced a novel algorithm to compute T-optimal experimental designs, the 2-ADAPT-MD. The algorithm utilizes a two-fold adaptive discretization of the parameters space $\Para_2$ as well as the design space $\X$.
Both discretizations are chosen similar to approaches from semi-infinite programming where an inner optimization problem is solved to select a new design point $x$, or a new parameter value $\para_2$, respectively.

We have proven the convergence of the 2-ADAPT-MD. Additionally we have shown bounds on the error in objective value, depending on the tolerance obtained in the discretized model discrimination sub-problem. The main advantage of the 2-ADAPT-MD is that it converges under only mild assumptions on the considered models. By using an increasing discretization in design points $x$ and parameters $\para_2$, we obtain linear programs and least-square programs as most frequent sub-problems. These structures are well studied and specialized solvers can be used to efficiently compute corresponding solutions. The use of discretizations also decreases our problem dimensions, which leads to a higher numerical stability. In particular, this is noticeable when adding a regularization to the least-square term.

In two examples from chemical engineering we have seen that the 2-ADAPT-MD can compete and even outperform state-of-the-art model discrimination methods. In comparison to a classical semi-infinite approach, we can drastically improve the time required for the computations. By using a second inner discretization in the parameters $\para_2$ we decrease the dimension in the regularization term of the least-squares problem.
Compared to the approach by Dette et al., we do not require linear approximations in the parameters for the selection of new candidate experiments. This makes our new algorithm more reliable and stable. Additionally, the 2-ADAPT-MD can also handle multi-response models, whereas the algorithm by Dette et al. is limited to the single-response case.

However, some aspects of the 2-ADAPT-MD need further investigation and can be improved upon. First, we need to solve a non-convex sub problem in each iteration to global optimality in order to refine the discretization in the design space $\X$. This is similar to most state-of-the-art algorithms, which rely on global optimization solvers or discrete design spaces $\X$.
Secondly, throughout this paper we have assumed the worst-case parameters $\hat{\para}_2$ to be unique. In future work we want to investigate if the 2-ADAPT-MD can be adjusted to also handle problems where the parameters are not unique.

%% file: Appendix/SIP.tex
\section{Semi-Infinite Programming}
\label{App:SIP}

We briefly introduce some basics of semi-infinite programming. For more detailed overviews concerning this topic we refer to \cite{lopez2007semi}, \cite{djelassi2021recent}, and \cite{shapiro2009semi}.


\begin{defi}[Semi-Infinite Program]
	A \textit{semi-infinite program} is an optimization problem of the form
	\begin{align*}
		\text{SIP}:~ &\max_{x\in\mathbb{R}^d} ~ f(x)\\
		&\text{ s.t.} ~g(x,y)\geq 0, ~\forall y\in Y,
	\end{align*}
	where $Y$ is an (infinite) set, called the \textit{index set}, and the corresponding variables $y\in Y$ are called \textit{index variables}.
\end{defi}

%

Generally speaking, a semi-infinite program is an optimization problem with infinitely many constraints which are induced by the index set $Y$. Similar to other types of optimization problems, one can consider special cases such as linear semi-infinite programs (LSIP).

\begin{defi}[Linear Semi-Infinite Program]
	A \textit{linear semi-infinite program} is a semi-infinite program of the form
	\begin{align*}
		\text{LSIP}:~&\max_{x\in\mathbb{R}^d} ~ c^T x\\
		&\text{ s.t.} ~ a(y)^T x\geq b(y), ~\forall y\in Y,
	\end{align*}
	where $a$ and $b$ can be arbitrary functions. In particular, they can be non-linear.
\end{defi}

Important examples of semi-infinite programs are maximin problems.

\begin{examplex}[Epigraph Formulation for Maximin Problems]
	\label{bsp:epi}
	A \textit{maximin problem}
	\[\max_{x\in X} \min_{y\in Y} \, f(x,y)\]
	can be formulated as a semi-infinite program in the following form:
	\begin{align*}
		\max_{\substack{x\in X\\t\in\mathbb{R}}} ~&t\\
		\text{s.t. } &f(x,y)\geq t, ~\forall y\in Y,
	\end{align*}
	which is also referred to as the \textit{epigraph form} of a maximin problem.
\end{examplex}

In order to find solutions for semi-infinite programs, several strategies have been introduced. Though the most prominent algorithm is the \textit{Blankenship \& Falk algorithm} \cite{blankenship1976infinitely}. We start its description by introducing the lower-level problem. For a given variable  $\bar{x}$, this sub-problem determines the index variable (or constraint) which is violated the most.

\begin{nota}[Lower-Level Problem]
	For a point $\bar{x}$, the \textit{lower-level problem} is given by
	\begin{align*}
		Q(\bar{x}):~ &\min_y ~ g(\bar{x}, y)\\
		&\text{ s.t. } ~ y\in Y.
	\end{align*}
\end{nota}
Notice that $\bar{x}$ is feasible, i.e.\ $g(x,y)\geq 0$ for all $y\in Y$, if the optimal solution $\bar{y}$ of $Q(\bar{x})$ satisfies $g(\bar{x}, \bar{y})\geq 0$.

However, it is in general still hard to derive a candidate for an optimal solution of a problem with infinitely many constraints. Therefore, we consider only a finite subset $\dot{Y}\subset Y$ of the index set which yields the \textit{discretized problem} or \textit{upper-level problem}.
\begin{nota}[Upper-Level Problem]
	For a given discretization $\dot{Y}\subset Y$ of the index set, the upper-level problem is given by
	\begin{align*}
		\text{SIP}(\dot{Y}):~ &\max_x ~ f(x)\\
		&\text{ s.t. } ~ g(x,y)\geq 0, ~\forall y\in \dot{Y}.
	\end{align*}
\end{nota}
Note that this discretized problem has finitely many constraints and can thus be solved by common strategies, e.g.\ non-linear programming.

Finally, we present the Blankenship \& Falk algorithm in Algorithm \ref{alg:BF}.

\begin{algorithm}[ht]
	\caption{Blankenship \& Falk Algorithm}\label{alg:BF}
	\begin{algorithmic}[1]\onehalfspacing
		\Require{$k=0$; initial discretization $Y\high{0}\subset Y$}
		\Repeat
		\State solve $\text{SIP}(Y\high{k})$ and obtain new solution $x\high{k}$ \Comment{upper-level}
		\State solve $Q(x\high{k})$ and obtain new solution $y\high{k}$ \Comment{lower-level}
		\If{$g(x\high{k}, y\high{k})\geq 0$} \Comment{check feasibility}
		\State \Return $x\high{k}$
		\EndIf
		\State $Y\high{k+1} := Y\high{k} \cup \{y\high{k}\}$ \Comment{augment discretization}
		\State $k := k+1$
		\Until{some stopping criterion is met}
	\end{algorithmic}
\end{algorithm}

One has to be careful with the computed solutions $x\high{k}$, though. They are not necessarily feasible because the upper-level problem in every iteration is a relaxation due to having fewer constraints. 
Nevertheless, convergence can be shown under some additional assumptions. 
For example the following theorem is stated in \cite{lopez2007semi}.

\begin{theorem}[Convergence of the Blankenship \& Falk Algorithm]
	\label{thm:convBF}
	Suppose that the starting feasible region $\mathcal{F}(Y\high{0}) := \{x \mid g(x,y) \geq 0, ~\forall y\in Y\high{0}\}$ is compact. Then, the Blankenship \& Falk algorithm either stops with an optimal solution of the given semi-infinite program, or any accumulation point of the sequence of solutions $(x\high{k})_{k\in\mathbb{N}}$ is an optimal solution.
\end{theorem}